\documentclass[11pt]{amsart}

\usepackage{amsmath,amssymb,latexsym,amsthm,newlfont,enumerate}
\usepackage{color}
\usepackage{tikz-cd}
\usepackage{hyperref}

\theoremstyle{plain}

\newtheorem{thm}{Theorem}[section]
\newtheorem{pro}[thm]{Proposition}

\newtheorem{lem}[thm]{Lemma}

\theoremstyle{definition}

\newtheorem{dfn}[thm]{Definition}
\newtheorem{rem}[thm]{Remark}

\newtheorem*{notation}{Notation}

\newcommand{\N}{\mathbb{N}}

\newcommand{\R}{\mathbb{R}}
\newcommand{\Q}{\mathbb{Q}}

\DeclareMathOperator{\sB}{\mathbf{B}}
\DeclareMathOperator{\Exc}{Exc}
\DeclareMathOperator{\mult}{mult}
\DeclareMathOperator{\Supp}{Supp}
\DeclareMathOperator{\nklt}{nklt}

\begin{document}
	\title[Special termination for log canonical pairs]{Special termination for log canonical pairs}

	\author{Vladimir Lazi\'c}
	\address{Fachrichtung Mathematik, Campus, Geb\"aude E2.4, Universit\"at des Saarlandes, 66123 Saarbr\"ucken, Germany}
	\email{lazic@math.uni-sb.de}
	
	\author{Joaqu\'in Moraga}
	\address{UCLA Mathematics Department, Box 951555, Los Angeles, CA 90095-1555, USA}
	\email{jmoraga@math.ucla.edu}

	\author{Nikolaos Tsakanikas}
	\address{EPFL SB MATH CAG, MA C3 595 (B\^atiment MA), Station 8, 1015 Lausanne, Switzerland}
	\email{nikolaos.tsakanikas@epfl.ch}
	
	\thanks{
		Lazi\'c was supported by the DFG-Emmy-Noether-Nachwuchsgruppe ``Gute Strukturen in der h\"oherdimensionalen birationalen Geometrie". We would like to thank O.\ Fujino and K.\ Hashizume for useful discussions related to this work.
	}
	
\begin{abstract}
We prove the special termination for log canonical pairs and its generalisation in the context of generalised pairs.
\end{abstract}
	
	\maketitle
	\setcounter{tocdepth}{1}
	\tableofcontents

\section{Introduction}

The first goal of this paper is to give a rigorous and complete proof of the following result, which is a natural step towards proving one of the big open problems in the Minimal Model Program (MMP) in characteristic zero -- the termination of flips.

\begin{thm}\label{thm:main}
	Assume the termination of flips for $ \Q $-factorial klt pairs of dimension at most $ n-1 $.
	
	Let $(X_1,B_1)$ be a quasi-projective log canonical pair of dimension $n$ which is projective over a normal quasi-projective variety $Z$. Consider a sequence of flips over $Z$:
	\begin{center}
		\begin{tikzcd}[column sep = 2em, row sep = 2.25em]
			(X_1,B_1) \arrow[dr, "\theta_1" swap] \arrow[rr, dashed, "\pi_1"] && (X_2,B_2) \arrow[dl, "\theta_1^+"] \arrow[dr, "\theta_2" swap] \arrow[rr, dashed, "\pi_2"] && (X_3,B_3) \arrow[dl, "\theta_2^+"] \arrow[rr, dashed, "\pi_3"] && \dots \\
			& Z_1 && Z_2
		\end{tikzcd}
	\end{center}
	Then there exists a positive integer $N$ such that
	\[ \Exc(\theta_i)\cap\nklt(X_i,B_i)=\emptyset \text{ for all } i\geq N . \]
\end{thm}

The result is usually referred to as \emph{special termination for log canonical pairs}; in other words, in any sequence of flips of a log canonical pair, the flipping (and thus also the flipped) locus will avoid the non-klt locus of the pair.

The theorem has its origins in \cite{Sho92} and was stated in this form in \cite{Sho03,Sho04}; however only a sketch of a proof in a special case is given. On the state of the art, see comments in \cite[Section 4.2]{Fuj07a} and \cite[Section 5]{Fuj11}.

The only complete proof of special termination is in \cite{Fuj07a} for \emph{dlt pairs}; in that case, special termination says that the flipping locus is eventually disjoint from the round-down of the boundary of the dlt pair. Even though that statement suffices in many applications, it seems that the generalisation to log canonical pairs is necessary if one wants to attack the termination of flips. In particular, special termination as in Theorem \ref{thm:main} was used in \cite{Bir07}.

Our second goal is to prove a form of special termination in the context of g-pairs. This is a recently introduced category which includes the usual pairs, see Section \ref{Sec:prelim} for details. 

\begin{thm}\label{thm:main_g}
	Assume the termination of flips for NQC $ \Q $-factorial klt g-pairs of dimension at most $ n-1 $. 
	
	Let $(X_1,B_1+M_1)$ be a quasi-projective NQC log canonical g-pair of dimension $n$ which is projective over a normal quasi-projective variety $Z$. Consider a sequence of flips over $Z$:
	\begin{center}
		\begin{tikzcd}[column sep = 0.8em, row sep = 1.75em]
			(X_1,B_1+M_1) \arrow[dr, "\theta_1" swap] \arrow[rr, dashed, "\pi_1"] && (X_2,B_2+M_2) \arrow[dl, "\theta_1^+"] \arrow[dr, "\theta_2" swap] \arrow[rr, dashed, "\pi_2"] && (X_3,B_3+M_3) \arrow[dl, "\theta_2^+"] \arrow[rr, dashed, "\pi_3"] && \dots \\
			& Z_1 && Z_2
		\end{tikzcd}
	\end{center}
	Then there exists a positive integer $N$ such that
	\[ \Exc(\theta_i)\cap\nklt(X_i,B_i+M_i)=\emptyset \text{ for all } i\geq N . \]
\end{thm}

As demonstrated in recent papers \cite{Mor18,HL22,HM20,LT22}, understanding the Minimal Model Program of g-pairs is indispensable even if one is interested only in results involving the usual pairs or even only varieties. We expect Theorem \ref{thm:main_g} to play a prominent role in future developments on the existence of minimal models and the termination of flips.

\section{Preliminaries}\label{Sec:prelim}

Throughout the paper we work over an algebraically closed field of cha\-rac\-te\-ris\-tic zero. All morphisms are projective.

Given a projective morphism $f\colon X\to Z $ between normal varieties and two $ \R $-Cartier divisors $ D_1 $ and $ D_2 $ on $ X $, we say that $ D_1 $ and $ D_2 $ are \emph{numerically equivalent over $ Z $}, denoted by $ D_1 \equiv_Z D_2 $, if $ D_1\cdot C = D_2 \cdot C $ for any curve $ C $ contained in a fibre of $ f $. We say that $ D_1 $ and $ D_2 $ are \emph{$\R$-linearly equivalent over $ Z $}, denoted by $ D_1 \sim_{\R,Z} D_2 $, if there exists an $\R$-Cartier $\R$-divisor $G$ on $Z$ such that $D_1\sim_\R D_2+f^*G$.

An $\R$-divisor $D$ on a variety $X$ over $Z$ is an \emph{NQC divisor} if it is a non-negative linear combination of $\Q$-Cartier divisors on $X$ which are nef over $ Z $. The acronym NQC stands for \emph{nef $\Q$-Cartier combinations}.

\begin{dfn}
	Let $X$ be a normal variety which is projective over a normal variety $Z$ and let $ D $ be an $ \R $-Cartier divisor on $ X $. An \emph{NQC weak Zariski decomposition} of $ D $ over $ Z $ consists of a projective birational morphism $ f \colon W \to X $ from a normal variety $ W $ and a numerical equivalence $ f^* D \equiv_Z P + N $ such that $ P $ is an NQC divisor and $ N $ is an effective $\R$-Cartier divisor on $W$. 
\end{dfn}

\begin{dfn}
	Let $X$ and $Y$ be normal varieties and let $\varphi \colon X\dashrightarrow Y$ be a \emph{birational contraction}, i.e.\ the map $ \varphi^{-1} $ contracts no divisors. Let $D$ be an $\R$-Cartier $\R$-divisor on $X$ and assume that $\varphi_*D$ is $\R$-Cartier. Then $\varphi$ is \emph{$D$-nonpositive} if there exists a smooth resolution of indeterminacies $(p,q)\colon W\to X\times Y$ of $\varphi$ such that 
	$$ p^*D\sim_\R q^* \varphi_*D + E,$$
	where $E$ is an effective $q$-exceptional $\R$-Cartier divisor on $W$.
\end{dfn}

\subsection{The relative stable base locus}

Let $ X \to Z $ be a projective morphism of normal varieties and let $ D $ be an $ \R $-divisor on $ X $. The \emph{$ \R $-linear system} associated with $ D $ over $ Z $ is 
		\[ |D/Z|_\R := \{ G \geq0 \mid G \sim_{\R,Z} D \} , \]
and the \emph{stable base locus} of $ D $ over $ Z $ is defined as
		\[ \sB(D) := \bigcap_{E \in |D/Z|_\R} \Supp E . \]

\begin{lem}\label{lem:baselocus}
	Let $f\colon Y\to X$ and $g\colon X\to Z$ be projective birational morphisms between normal varieties. Let $D$ be an $\R$-Cartier $\R$-divisor on $X$. Then $\sB(f^*D/Z)=f^{-1}\big(\sB(D/Z)\big)$.
\end{lem}	

\begin{proof}
	It suffices to show that $f^*|D/Z|_\R=|f^*D/Z|_\R$. It is clear that $f^*|D/Z|_\R\subseteq |f^*D/Z|_\R$. For the converse inclusion, let $ G \in | f^*D/Z |_\R $. We may write $ f^* f_* G = G + E $, where $ E $ is $ f $-exceptional. There exists an $ \R $-Cartier $\R$-divisor $ L $ on $ Z $ such that
	\[ G \sim_\R f^*D + (g \circ f)^*L = f^* ( D+g^*L ) , \]
	and thus
	\[ E \sim_\R f^* ( f_*G - D - g^*L ) . \]
	Therefore $ E = 0 $ by the Negativity Lemma \cite[Lemma 3.39(1)]{KM98}, and consequently $ f^* f_* G = G $. This proves that $f^*|D/Z|_\R\supseteq |f^*D/Z|_\R$ and completes the proof.
\end{proof}

\subsection{Generalised pairs}

For the definitions and basic results on the singularities of pairs and the MMP we refer to \cite{KM98}. Below we discuss briefly generalised pairs, abbreviated as g-pairs; for futher information we refer to \cite[Section 4]{BZ16}, and in particular to \cite[\S 2.1 and \S 3.1]{HL22} for properties of dlt g-pairs. 

\begin{dfn}
	A \emph{generalised pair} or \emph{g-pair} $(X,B+M)$ consists of a normal variety $ X $ equipped with  projective morphisms 
	$$  X' \overset{f}{\longrightarrow} X \longrightarrow Z , $$ 
	where $ f $ is birational and $ X' $ is normal, $B$ is an effective $ \R $-divisor on $X$, and $M'$ is an $\R$-Cartier divisor on $X'$ which is nef over $Z $ such that $ f_* M' = M $ and $ K_X + B + M $ is $ \R $-Cartier. We say often that the g-pair $(X,B+M)$ is given by the data $X'\to X\to Z$ and $M'$.
	
	Moreover, if $M'$ is an NQC divisor on $ X' $, then the g-pair $(X/Z,B+M)$ is an \emph{NQC g-pair}.
	
	Finally, we say that a g-pair $ (X,B+M) $ admits an NQC weak Zariski decomposition over $ Z $ if the divisor $ K_X + B + M $ has an NQC weak Zariski decomposition over $ Z $.
\end{dfn}

For simplicity, we denote such a g-pair only by $ (X/Z,B+M) $, but we implicitly remember the whole g-pair structure. Additionally, we note that the definition is flexible with respect to $X'$ and $M'$: if $ g \colon Y \to X' $ is a projective birational morphism from a normal variety $ Y $, then we may replace $X'$ with $Y$ and $M'$ with $ g^*M'$. Hence, we may always assume that $f \colon X'\to X$ in the above definition is a sufficiently high birational model of $ X $. 

\begin{dfn}
	Let $ (X,B+M) $ be a g-pair with data $ X' \overset{f}{\to} X \to Z $ and $ M' $. We can then write 
	$$ K_{X'} + B' + M' \sim_\R f^* ( K_X + B + M ) $$
	for some $ \R $-divisor $ B' $ on $ X' $. Let $ E $ be a divisorial valuation over $ X $ which is a prime divisor on $ X' $; its centre on $X$ is denoted by $c_X(E)$. The \emph{discrepancy of $ E $} with respect to $ (X,B+M) $ is $ a (E, X, B+M) := {-} \mult_E B' $.
	The g-pair $ (X,B+M) $ is: 
	\begin{enumerate}
		\item[(a)] \emph{klt} if $a (E, X, B+M) > -1 $ for all divisorial valuations $E$ over $X$,
		\item[(b)] \emph{log canonical} if $a (E, X, B+M) \geq -1 $ for all divisorial valuations $E$ over $X$,
		\item[(c)] \emph{dlt} if it is log canonical, if there exists an open subset $U\subseteq X$ such that the pair $(U,B|_U)$ is log smooth, and if $a(E,X,B+M) = {-}1$ for some divisorial valuation $E$ over $X$, then the set $c_X(E)\cap U$ is non-empty and it is a log canonical centre of $(U,B|_U)$.
	\end{enumerate}
	
		If $ (X,B+M) $ is a log canonical g-pair, then:
	\begin{enumerate}
		\item[(i)] an irreducible subvariety $ S $ of $ X $ is a \emph{log canonical centre} of $ (X,B+M) $ if there exists a divisorial valuation $ E $ over $ X $ such that $ c_X(E) = S $ and $ a(E,X,B+M) = -1 $,
		
		\item[(ii)] the \emph{non-klt locus} of $ (X,B+M) $, denoted by $ \nklt(X,B+M) $, is the union of all log canonical centres of $ (X,B+M) $.
	\end{enumerate}
	
	When $M'=0$, one recovers the definitions of singularities of usual pairs.
\end{dfn}

It is clear from the definition that if $ (X,B+M) $ is a dlt g-pair with $ \lfloor B \rfloor = 0 $, then $ (X,B+M) $ is klt.

	If $ (X,B+M) $ is a $ \Q $-factorial dlt g-pair, then by definition and by \cite[Remark 4.2.3]{BZ16}, the underlying pair $ (X,B) $ is $ \Q $-factorial dlt and the log canonical centres of $ (X,B+M) $ coincide with those of $ (X,B) $. In particular 
	\[ \nklt(X,B+M) = \nklt(X,B) = \Supp \lfloor B \rfloor \]
	by \cite[Proposition 3.9.2]{Fuj07a}. We will use this repeatedly in the paper without explicit mention.

We adopted the definition dlt g-pairs from \cite{HL22}, which behaves well under restrictions to log canonical centres and under operations of an MMP; such operations are analogous to those in the standard setting, see \cite[Section 4]{BZ16} or \cite[\S3.1]{HL22} for details.

\begin{notation}
	We will use the following notation throughout the paper. Let $(X,B+M)$ be a dlt g-pair and let $S$ be a log canonical centre of $(X,B+M)$. We define a dlt g-pair $(S,B_S+M_S)$ by adjunction, i.e.\ by the formula 
	$$K_S + B_S +M_S = (K_X + B +M)|_S$$
	as in \cite[Proposition 2.8]{HL22}. 
\end{notation}

The next result is \cite[Proposition 3.9]{HL22} and is used frequently in the paper.

\begin{lem}
	Let $(X,B+M)$ be a log canonical g-pair with data $ X' \overset{f}{\to} X \to Z $ and $ M' $. Then, after possibly replacing $f$ with a higher model, there exist a $\Q$-factorial dlt g-pair $(Y,\Delta+N)$ with data $ X' \overset{g}{\to} Y \to Z $ and $ M' $, and a projective birational morphism $\pi \colon Y \to X$ such that 
	$$K_Y + \Delta + N \sim_\R \pi^*(K_X + B +M).$$
	The g-pair $(Y,\Delta+N)$ is a \emph{dlt blowup} of $(X,B+M)$.
\end{lem}

\begin{rem}\label{rem:nklt_dlt_blow-up}
	If $ f \colon (Y,\Delta+N) \to (X,B+M) $ is a dlt blowup of a log canonical g-pair $ (X,B+M) $, then by a suitable analogue of \cite[Lemma 2.30]{KM98} we have
	$$ \nklt(X,B+M) = f\big(\nklt(Y,\Delta+N)\big) = f\big( \Supp \lfloor \Delta \rfloor \big). $$ 
In particular, the number of log canonical centres of a given log canonical g-pair is finite.
\end{rem}

\subsection{Monotonicity of discrepancies}

Parts (i) and (iii) of the following result are a version of the so-called monotonicity lemma for g-pairs. Parts (ii) and (iv) will also be needed below, when we deal with the \emph{difficulty} of g-pairs.

\begin{lem} \label{lem:discrep}
	Let $(X,B+M)$ and $(X',B'+M')$ be g-pairs such that there exists a diagram 
		\begin{center}
		\begin{tikzcd}
			& Z \arrow[dl, "g" swap] \arrow[dr, "g'"] \\
			X \arrow[rr, "\varphi", dashed] \arrow[dr, "f" swap] && X' \arrow[dl, "f'"] \\
			& Y ,
		\end{tikzcd}
	\end{center}
	where $Y$ and $Z$ are normal varieties, all morphisms are proper birational, $ K_{X'} + B' + M' $ is $ f' $-nef, and there exists a nef $ \R $-Cartier $ \R $-divisor $ M_Z $ on $ Z $ with $ g_* M_Z = M $ and $ g'_* M_Z = M' $.
	\begin{enumerate}
		\item[(i)] Assume that $B'=\varphi_*B+E$, where $E$ is the sum of all prime divisors which are contracted by $\varphi^{-1}$, and that
		$$a(F,X,B+M)\leq a(F,X',B'+M')$$ 
		for every $\varphi$-exceptional divisor $F$ on $X$. Then for any geometric valuation $ F $ over $ X $ we have
	\[ a(F,X,B+M) \leq a(F,X',B'+M') . \]
	
		\item[(ii)] Under the assumptions of (i), assume additionally that $(X,B+M)$ is dlt and let $S$ be a log canonical centre of $(X,B+M)$. Assume that $\varphi$ is an isomorphism at the generic point of $S$ and define $S'$ as the strict transform of $S$ on $X'$. Then for any geometric valuation $ F $ over $ S $ we have
	\[ a(F,S,B_S+M_S) \leq a(F,S',B_{S'}+M_{S'}) . \]

		\item[(iii)] Assume that $ -(K_X + B + M) $ is $ f $-nef and that $ f_* B = f'_* B' $. Then for any geometric valuation $ F $ over $ Y $ we have
	\[ a(F,X,B+M) \leq a(F,X',B'+M') , \]
	and strict inequality holds if either 
	\begin{enumerate}
		\item[(a)] $ -(K_X + B + M) $ is $ f $-ample and $ f $ is not an isomorphism above the generic point of $ c_Y(F) $, or
		
		\item[(b)] $ K_{X'} + B' + M' $ is $ f' $-ample and $ f' $ is not an isomorphism above the generic point of $ c_Y(F) $.
	\end{enumerate}
	In particular, if $(X,B+M)$ is log canonical and if either (a) or (b) holds, then $F$ is not a log canonical centre of $(X',B'+M')$.

		\item[(iv)] Assume that $ -(K_X + B + M) $ is $ f $-nef, that $ f_* B = f'_* B' $ and that $(X,B+M)$ is dlt. Let $S$ be a log canonical centre of $(X,B+M)$, assume that $\varphi$ is an isomorphism at the generic point of $S$ and define $S'$ as the strict transform of $S$ on $X'$. Let $T$ be the normalisation of $f(S)$, so that we have the following diagram:
\begin{center}
	\begin{tikzcd}[column sep = 2em, row sep = 2.25em]
		(S,B_S+M_S) \arrow[dr, "f|_S" swap] \arrow[rr, dashed, "\varphi|_S"] && (S',B_{S'}+M_{S'}) \arrow[dl, "f'|_{S'}"] \\
		& T. &
	\end{tikzcd}
\end{center}
 Then for any geometric valuation $ F $ over $ T $ we have
	\[ a(F,S,B_S+M_S) \leq a(F,S',B_{S'}+M_{S'}) , \]
	and strict inequality holds if either 
	\begin{enumerate}
		\item[(a)] $ -(K_X + B + M) $ is $ f $-ample and $ f|_S $ is not an isomorphism above the generic point of $ c_T(F) $, or
		
		\item[(b)] $ K_{X'} + B' + M' $ is $ f' $-ample and $ f'|_S $ is not an isomorphism above the generic point of $ c_T(F) $.
	\end{enumerate}
	In particular, if either (a) or (b) holds, then $F$ is not a log canonical centre of $(S',B_{S'}+M_{S'})$.
	\end{enumerate}	
	\end{lem}

\begin{proof}
	The proofs of (i) and (iii) are analogous to the proofs of \cite[Lemma 3.38 and Proposition 3.51]{KM98} and we provide the details for the benefit of the reader.
	
	By possibly replacing $ Z $ by a higher birational model, we may additionally assume that $ c_Z(F) $ is a divisor on $ Z $. Set $ h := f \circ g = f' \circ g' $. Then
	\begin{align*}
		K_Z + M_Z &\sim_\R g^* (K_X + B + M) + \sum a(F_i,X,B+M) F_i \\
		&\sim_\R (g')^* (K_{X'} + B' + M') + \sum a(F_i,X',B'+M') F_i .
	\end{align*}	 
	Consider the $ \R $-Cartier $ \R $-divisor 
	\begin{align} 
	H&:=\sum \big( a(F_i,X,B+M) - a(F_i,X',B'+M') \big) F_i \label{eq:78}\\
	& \sim_\R (g')^* (K_{X'} + B' + M') - g^* (K_X + B + M).\notag
	\end{align}
	Under the assumptions of (i) the divisor $H$ is $g$-nef and $g_*H\leq0$, hence $H\leq0$ by \cite[Lemma 3.39(1)]{KM98}. Under the assumptions of (iii) the divisor $H$ is $ h $-nef and $ h_*H=0 $ since $ f_* B = f'_* B' $, hence $H\leq0$ by \cite[Lemma 3.39(1)]{KM98}. This yields (i) and the first statement of (iii).
	
	If the case (a) or the case (b) of (iii) holds, then $ H $ is not numerically $ h $-trivial over the generic point $\eta$ of $ c_Y(F) $. Then \cite[Lemma 3.39(2)]{KM98} implies that $ h^{-1}(\eta) \subseteq \Supp H $ and therefore that $ F \subseteq \Supp H $. This yields the second statement of (iii).
	
	For (ii) and (iv), by \cite[Lemma 2.45]{KM98} there is a sequence of blowups of $S$ along the centres of $F$ such that the centre of $F$ becomes a divisor. By considering these blowups as blowups of $X$ and possibly blowing up further, we may assume that the centre $c_{S_Z}(F)$ is a divisor, where $S_Z$ is the strict transform of $S$ on $Z$, and that there exist finitely many prime divisors $\widehat{F}_i$ on $Z$ such that $\widehat{F}_i|_{S_Z}=c_{S_Z}(F)$ and $c_T(F)=c_Y\big(\widehat{F}_i\big)|_T$ for each such $\widehat{F}_i$. Then (ii) follows by restricting the relation \eqref{eq:78} to $S_Z$: indeed, $S_Z\nsubseteq \Supp H$ since $\varphi$ is an isomorphism at the generic point of $S$.
	
	In the case (a) of (iv) we have then that $f$ is not an isomorphism above the generic point of each $c_Y\big(\widehat{F}_i\big)$, so (iv) follows from (iii) as in the proof of (ii) above. We obtain (iv) in the case (b) analogously, by first blowing up along the centres of $F$ on $S'$ instead.
\end{proof}

\subsection{Minimal models and canonical models}

\begin{dfn} 
	Let $(X,B +M)$ be a log canonical g-pair with data $ X' \overset{f}{\to} X \to Z $ and $ M' $ and consider a birational map $ \varphi \colon (X,B+M) \dashrightarrow (Y,B_Y+M_Y)$ over $Z$ to a $\Q$-factorial g-pair $ (Y/Z,B_Y+M_Y) $. We may assume that $X'$ is a high enough model so that the map $\varphi\circ f$ is a morphism. Then $\varphi$ is a \emph{minimal model in the sense of Birkar-Shokurov over $Z$} of the g-pair $(X,B+M)$ if $ B_Y =\varphi_*B+E$, where $E$ is the sum of all prime divisors which are contracted by $\varphi^{-1}$, if $M_Y=(\varphi\circ f)_*M'$, if the divisor $K_Y+B_Y+M_Y$ is nef over $Z$ and if
	$$a(F,X,B+M) < a(F,Y,B_Y+M_Y)$$
	for any prime divisor $ F $ on $ X $ which is contracted by $\varphi $. Note that then the g-pair $ (Y,B_Y+M_Y)$ is log canonical by Lemma \ref{lem:discrep}(i).
	
	If, moreover, the map $\varphi$ is a birational contraction, but $Y$ is not necessarily $\Q$-factorial if $X$ is not $\Q$-factorial (and $Y$ is $\Q$-factorial if $X$ is $\Q$-factorial), then $\varphi$ is a \emph{minimal model} of $(X/Z,B+M)$.
\end{dfn}

For the differences among these notions of a minimal model, see \cite[\S 2.2]{LT22}; note that here we allow a minimal model in the sense of Birkar-Shokurov to be log canonical and not only dlt, which is in alignment with the definitions in \cite{Hash18a,Hash19b}.

\begin{rem}\label{rem:models}
Let $(X,B +M)$ be a log canonical g-pair and let $ \varphi \colon (X,B+M) \dashrightarrow (Y,B_Y+M_Y)$ be a minimal model of $(X,B +M)$ over $Z$. Then any dlt blowup of $(Y,B_Y+M_Y)$ is a minimal model in the sense of Birkar-Shokurov of $(X,B +M)$ over $Z$.
\end{rem}

\begin{dfn}
	Assume that we have a commutative diagram
	\begin{center}
		\begin{tikzcd}				(X,B+M) \arrow[rr, "\varphi", dashed] \arrow[dr, "f" swap] && (X',B'+M') \arrow[dl, "f'"] \\
			& Y 
		\end{tikzcd}
	\end{center}
	where $(X,B+M)$ and $(X',B'+M')$ are g-pairs and $Y$ is normal, $f$ and $f'$ are projective birational morphisms and $\varphi$ is a birational contraction, $ \varphi_* B = B' $, and $ M $ and $ M' $ are pushforwards of the same nef $\R$-divisor on a common birational model of $X$ and $X'$. If $(X,B+M)$ is log canonical, if $ K_{X'} + B' + M' $ is ample over $Y$ and if $a(F,X,B+M)\leq a(F,X',B'+M')$ for every $\varphi$-exceptional prime divisor $F$ on $X$, then $(X',B'+M')$ is a \emph{log canonical model of $(X,B+M)$}.
\end{dfn}

The following result shows how a minimal model and a log canonical model of a g-pair are related.

\begin{lem} \label{lem:maptoCM}
	Let $ (X/Z,B+M) $ be a log canonical g-pair, let 
	$ (X^m,B^m+M^m) $ be a minimal model of $ (X,B+M) $ over $ Z $ and let $ (X^\textit{lc},B^\textit{lc}+M^\textit{lc}) $ be a log canonical model of $ (X,B+M) $ over $ Z $. Then there exists a birational morphism 
	$ \alpha \colon X^m \to X^\textit{lc} $ such that
	\[ K_{X^m} + B^m+M^m \sim_\R \alpha^* (K_{X^\textit{lc}} + B^\textit{lc}+M^\textit{lc}) . \]
	In particular, $ K_{X^m} + B^m+M^m $ is semiample over $ Z $ and there exists a unique log canonical model of $(X,B+M)$, up to isomorphism. 
\end{lem}

\begin{proof}
	The proof is analogous to the proof of \cite[Lemma 4.8.4]{Fuj17} and we provide the details for the benefit of the reader.
	
	Let $W$ be a common resolution of $X$, $X^m$ and $X^\textit{lc}$, together with morphisms $p\colon W\to X$, $q\colon W\to X^m$ and $r\colon W\to X^\textit{lc}$. We may write
	\[ p^*(K_X + B + M) \sim_\R q^*(K_{X^m} + B^m+M^m) + F \]
	and
	\[ p^*(K_X + B + M) \sim_\R r^*(K_{X^\textit{lc}} + B^\textit{lc}+M^\textit{lc}) + G , \]
	where $ F$ is effective and $ q $-exceptional and $ G$ is effective and $ r $-exceptional, see Lemma \ref{lem:discrep}. Therefore,
	\[ q^*(K_{X^m} + B^m+M^m) + F \sim_\R r^*(K_{X^\textit{lc}} + B^\textit{lc}+M^\textit{lc}) + G . \]
	Note that $ q_*(G-F)\geq0 $ and $ -(G-F) $ is $ q $-nef, and that $ r_*(F-G)\geq0 $ and $ -(F-G) $ is $ r $-nef. This implies that $ F=G $ by the Negativity lemma \cite[Lemma 3.39]{KM98}, and therefore,
	\begin{equation}\label{eq:rig}
		q^*(K_{X^m} + B^m+M^m) = r^*(K_{X^\textit{lc}} + B^\textit{lc}+M^\textit{lc}).
	\end{equation}
	Let $ C $ be a curve on $ W $ which is contracted by $ q $. Then
	\begin{align*}
		0 &= q^*(K_{X^m} + B^m+M^m) \cdot C = r^*(K_{X^\textit{lc}} + B^\textit{lc}+M^\textit{lc}) \cdot C \\
		&= (K_{X^\textit{lc}} + B^\textit{lc}+M^\textit{lc}) \cdot r_* C , 
	\end{align*}
	hence $ C $ is contracted by $ r $ as $ K_{X^\textit{lc}} + B^\textit{lc}+M^\textit{lc} $ is ample over $ Z $. Thus, by the Rigidity lemma \cite[Lemma 1.15]{Deb01} there exists a birational morphism $ \alpha \colon X^m \to X^\textit{lc} $ such that $ r=\alpha \circ q $, and the first statement follows from \eqref{eq:rig}.
	
	Assume that there exists another log canonical model $(Y,B_Y+M_Y)$ of $(X,B+M)$. Then analogously as above, there exists a birational morphism $\beta\colon X^\textit{lc}\to Y$ such that
	\[ K_{X^\textit{lc}} + B^\textit{lc}+M^\textit{lc} \sim_\R \beta^* (K_{Y} + B_Y+M_Y) . \]
 	Since the divisor $K_{X^\textit{lc}} + B^\textit{lc}+M^\textit{lc}$ is ample over $Z$, the map $\beta$ must be an isomorphism.
\end{proof}

\begin{dfn}
	Assume that we have a commutative diagram
		\begin{center}
			\begin{tikzcd}
				(X,B+M) \arrow[rr, "\varphi", dashed] \arrow[dr, "f" swap] && (X',B'+M') \arrow[dl, "f'"] \\
				& Y 
			\end{tikzcd}
		\end{center}
	where $(X,B+M)$ and $(X',B'+M')$ are g-pairs and $Y$ is normal, $f$ and $f'$ are projective birational morphisms and $\varphi$ is an isomorphism in codimension one, $ \varphi_* B = B' $, and $ M $ and $ M' $ are pushforwards of the same nef $\R$-divisor on a common birational model of $X$ and $X'$. 
	\begin{enumerate}
		\item[(a)] If $ -(K_X + B + M) $ is ample over $Y$ and if $ K_{X'} + B' + M' $ is ample over $Y$, then the diagram is an \emph{ample small quasi-flip}.
	
		\item[(b)] An ample small quasi-flip is a \emph{flip} if $f$ and $f'$ are isomorphisms in codimension one and if $ \rho(X/Y) = \rho(X'/Y) =1 $.
	\end{enumerate}
\end{dfn}

We recall that flips for log canonical pairs exist by \cite[Corollary 1.2]{Bir12a} or \cite[Corollary 1.8]{HX13}. On the other hand, the existence of flips for g-pairs is not known yet in full generality. We refer, however, to \cite[Section 4]{BZ16} and \cite[\S 3.1]{HL22} for known special cases.

The following result follows from \cite[Lemmas 4.3.8 and 4.9.3]{Fuj17}.

\begin{lem} \label{lem:klt_lc}
	The termination of flips for $ \Q $-factorial klt pairs of dimension at most $ d $ implies the termination of flips for log canonical pairs of dimension $ d $.
\end{lem}

\subsection{The difficulty}

In this subsection we follow closely \cite{HL22} and we include the details for the benefit of the reader. The \emph{difficulty} stands as a collective noun for various invariants related to the discrepancies of a pair or a g-pair, which behave well under the operations of the Minimal Model Program. The first version of the difficulty was introduced in \cite{Sho85}. The version below was defined in \cite{HL22}.

\begin{dfn}\label{dfn:difficulty}
Let $(X,B + M)$ be an NQC $\Q$-factorial dlt g-pair with data $ X' \to X \to Z $ and $ M' $. We may write $B=\sum_{i=1}^k b_i B_i$ with prime divisors $B_i$ and $b_i\in (0,1]$, and 
$ M' =\sum_{i=1}^\ell \mu_i M_i'$ with $ M_i'$ Cartier divisors which are nef over $ Z $ and $ \mu_i \in (0, +\infty) $.
Let $S$ be a log canonical centre of $(X,B +M)$. Set $b = \{b_1,\dots,b_k\}$, $\mu = \{\mu_1,\dots,\mu_\ell\}$, and 
$$\mathcal S(b,\mu) = \bigg\{\frac{m-1}{m} + \sum_{i=1}^k\frac{r_ib_i}{m} + \sum_{i=1}^\ell\frac{s_i\mu_i}{m}\leq 1 \ \Big|\ m \in \N_{>0}, r_i , s_i \in\N\bigg\}.$$
Note that the coefficients of $B_S$ belong to the set $\mathcal S(b,\mu)$ by the proof of \cite[Proposition 4.9]{BZ16}. For each $\alpha\in\mathcal S(b,\mu)$ set
$$ d_{<-\alpha}(S,B_S +M_S) = \#\big\{E \mid a(E,S,B_S +M_S) < {-} \alpha, \ c_S(E) \nsubseteq \Supp \lfloor B_S \rfloor \big\}$$
and
$$ d_{\leq -\alpha}(S,B_S +M_S) = \#\big\{E \mid a(E,S,B_S +M_S) \leq {-} \alpha, \ c_S(E) \nsubseteq \Supp \lfloor B_S \rfloor \big\}.$$
The \emph{difficulty} of the g-pair $(S,B_S + M_S)$ is defined as
$$ d_{b,\mu}(S,B_S +M_S) = \sum_{\alpha\in\mathcal S(b,\mu)} \Big(d_{<-\alpha}(S,B_S +M_S)+d_{\leq-\alpha}(S,B_S +M_S)\Big).$$
\end{dfn}

\begin{lem}\label{lem:finite}
	In the notation from Definition \ref{dfn:difficulty}:
	\begin{enumerate}
		\item[(i)] there exists $\gamma\in(0,1)$ such that $a(E,S,B_S +M_S)\geq-\gamma$ for each geometric valuation $E$ over $S$ such that $c_S(E) \nsubseteq \Supp \lfloor B_S \rfloor$;
		\item[(ii)] the set $\mathcal S(b,\mu)\cap[0,\gamma]$ is finite;
		\item[(iii)] we have 
			$$d_{b,\mu}(S,B_S +M_S) <+\infty.$$
		\end{enumerate}	
\end{lem}

\begin{proof}
	Let $S'\stackrel{\sigma}{\to} S\to Z$ and $M'_S$ be the data of the g-pair $(S,B_S+M_S)$. Consider the set $U:=S\setminus\Supp\lfloor B_S\rfloor$ and let $U':=\sigma^{-1}(U)$. Then we obtain the klt g-pair $(U,B_S|_U+M_S|_U)$ with data $U'\to U\to U$ and $M'_S|_{U'}$. Define $B'_S$ by the equation
	$$K_{S'}+B'_S+M'_S\sim_\R \sigma^*(K_S+B_S+M_S).$$
	Then for each geometric valuation $E$ over $S$ such that $c_S(E) \nsubseteq \Supp \lfloor B_S \rfloor$ we have 
	\begin{align*}
	a(E,S,B_S +M_S)&=a(E,U,B_S|_U +M_S|_U)\\
	&=a(E,U',B'_S|_{U'} +M'_S|_{U'})=a(E,U',B'_S|_{U'}),
	\end{align*}
	hence
	\begin{equation}\label{eq:55}
	d_{<-\alpha}(S,B_S +M_S) = \#\big\{E \mid a(E,U',B'_S|_{U'}) < {-} \alpha\big\}	
	\end{equation}
and
	\begin{equation}\label{eq:56}
	d_{\leq -\alpha}(S,B_S +M_S) = \#\big\{E \mid a(E,U',B'_S|_{U'}) \leq {-} \alpha\big\}.	
	\end{equation}
	Since the pair $(U',B'_S|_{U'})$ is klt, by \eqref{eq:55} and \eqref{eq:56} there exists $\gamma\in(0,1)$ such that (i) holds, and in particular:
	$$d_{<-\alpha}(S,B_S +M_S) =d_{\leq-\alpha}(S,B_S +M_S) =0\quad \text{if }\alpha>\gamma.$$
	On the other hand, $d_{<-\alpha}(S,B_S +M_S) $ and $d_{\leq-\alpha}(S,B_S +M_S) $ are finite for any $\alpha\in\mathcal S(b,\mu)$ by \eqref{eq:55} and \eqref{eq:56} and by \cite[Proposition 2.36(2)]{KM98}. Since the set $\mathcal S(b,\mu)\cap[0,\gamma]$ is finite by \cite[Lemma 7.4.4]{Kol92}, (ii) and (iii) follow.
	\end{proof}

\begin{pro}\label{pro:monoton}
Assume the notation from Definition \ref{dfn:difficulty}. Consider a flip
	\begin{center}
		\begin{tikzcd}[column sep = 0.8em, row sep = 1.75em]
			(X,B+M) \arrow[dr, "\theta" swap] \arrow[rr, dashed, "\pi"] && (X^+,B^++M^+) \arrow[dl, "\theta^+"] \\
			& Z &
		\end{tikzcd}
	\end{center}
Assume that $\pi$ is an isomorphism at the generic point of $S$ and define $S^+$ as the strict transform of $S$ on $X^+$. Moreover, assume that $\pi|_S$ is an isomorphism along $\Supp \lfloor B_S\rfloor$. Then the following hold.
	\begin{enumerate}
		\item[(i)] We have
			$$ d_{b,\mu}(S,B_S +M_S) \geq d_{b,\mu}(S^+,B_{S^+} +M_{S^+}). $$
		\item[(ii)] If there exists a geometric valuation $E$ over $S$ such that $c_S(E)$ is a divisor but $c_{S^+}(E)$ is not a divisor, then there exists $\alpha_0\in\mathcal S(b,\mu)\setminus\{1\}$ such that
			$$d_{\leq-\alpha_0}(S,B_S +M_S) > d_{\leq-\alpha_0}(S^+,B_{S^+} +M_{S^+}).$$
		\item[(iii)] If there exists a geometric valuation $E$ over $S$ such that $c_S(E)$ is not a divisor but $c_{S^+}(E)$ is a divisor, then there exists $\alpha_0\in\mathcal S(b,\mu)\setminus\{1\}$ such that
			$$d_{<-\alpha_0}(S,B_S +M_S) > d_{<-\alpha_0}(S^+,B_{S^+} +M_{S^+}).$$
		\item[(iv)] If $\pi|_S$ is not an isomorphism in codimension $1$, then
			$$ d_{b,\mu}(S,B_S +M_S) > d_{b,\mu}(S^+,B_{S^+} +M_{S^+}). $$
	\end{enumerate}		
\end{pro}

\begin{proof}
Part (i) follows immediately from Lemma \ref{lem:discrep}.

For (ii), note that $c_S(E)\not\subseteq\Supp\lfloor B_S\rfloor$ and $c_{S^+}(E)\not\subseteq\Supp\lfloor B_{S^+}\rfloor$ since $\pi|_S$ is an isomorphism along $\Supp \lfloor B_S\rfloor$. Then there exists $\alpha_0\in\mathcal S(b,\mu)\setminus\{1\}$ such that, by Lemma \ref{lem:discrep},
$$-\alpha_0=a(E,S,B_S +M_S)< a(E,S^+,B_{S^+} +M_{S^+}),$$
and (ii) follows.

For (iii), as above we again have $c_S(E)\not\subseteq\Supp\lfloor B_S\rfloor$ and $c_{S^+}(E)\not\subseteq\Supp\lfloor B_{S^+}\rfloor$. Then there exists $\alpha_0\in\mathcal S(b,\mu)\setminus\{1\}$ such that, by Lemma \ref{lem:discrep},
$$a(E,S,B_S +M_S)< a(E,S^+,B_{S^+} +M_{S^+})=-\alpha_0,$$
and (iii) follows.

Part (iv) is an immediate consequence of (ii) and (iii).
\end{proof}

\section{Lifting a sequence of quasi-flips}

The following result allows us to pass from a sequence of (quasi-)flips of log canonical pairs to a sequence of flips of dlt pairs.

\begin{lem}\label{lem:lifting}
Let $ (X_1,B_1) $ be a quasi-projective log canonical pair over a quasi-projective variety $Z$. Consider a sequence of small ample quasi-flips over $Z$:
\begin{center}
	\begin{tikzcd}[column sep = 2em, row sep = 2.25em]
		(X_1,B_1) \arrow[dr, "\theta_1" swap] \arrow[rr, dashed, "\pi_1"] && (X_2,B_2) \arrow[dl, "\theta_1^+"] \arrow[dr, "\theta_2" swap] \arrow[rr, dashed, "\pi_2"] && (X_3,B_3) \arrow[dl, "\theta_2^+"] \arrow[rr, dashed, "\pi_3"] && \dots \\
		& Z_1 && Z_2
	\end{tikzcd}
\end{center}
Then there exists a diagram
\begin{center}
	\begin{tikzcd}[column sep = 2em, row sep = large]
		(Y_1,\Delta_1) \arrow[d, "f_1" swap] \arrow[rr, dashed, "\rho_1"] && (Y_2,\Delta_2) \arrow[d, "f_2" swap] \arrow[rr, dashed, "\rho_2"] && (Y_3,\Delta_3) \arrow[d, "f_3" swap] \arrow[rr, dashed, "\rho_3"] && \dots 
		\\ 
		(X_1,B_1) \arrow[dr, "\theta_1" swap] \arrow[rr, dashed, "\pi_1"] && (X_2,B_2) \arrow[dl, "\theta_1^+"] \arrow[dr, "\theta_2" swap] \arrow[rr, dashed, "\pi_2"] && (X_3,B_3) \arrow[dl, "\theta_2^+"] \arrow[rr, dashed, "\pi_3"] && \dots 
		\\
		& Z_1 && Z_2
	\end{tikzcd}
\end{center}
where, for each $i \geq 1$, the map $ \rho_i \colon Y_i \dashrightarrow Y_{i+1} $ is a $(K_{Y_i}+\Delta_i)$-MMP over $ Z_i $ and the map $f_i $ is a dlt blowup of the pair $ (X_i,B_i) $. 

In particular, the sequence on top of the above diagram is an MMP for a $\Q$-factorial dlt pair $ (Y_1,\Delta_1) $.
\end{lem}

\begin{proof}
Let $ f_1 \colon (Y_1,\Delta_1) \to  (X_1,B_1) $ be a dlt blowup of $ (X_1,B_1) $. By Remark \ref{rem:models} the pair $ (X_1,B_1) $ has a minimal model in the sense of Birkar-Shokurov over $ Z_1 $, hence $ (Y_1,\Delta_1) $ has a minimal model in the sense of Birkar-Shokurov over $ Z_1 $ by \cite[Lemma 2.15]{Hash19b}. Therefore, by \cite[Theorem 1.9(ii),(iii)]{Bir12a} there exists a $ (K_{Y_1} + \Delta_1) $-MMP with scaling of an ample divisor over $Z_1$ which terminates with a minimal model $ (Y_2,\Delta_2) $ of $ (Y_1,\Delta_1) $ over $Z_1$. Since $(X_2,B_2)$ is a log canonical model of $(Y_1,\Delta_1)$ over $ Z_1 $, by Lemma \ref{lem:maptoCM} there exists a morphism $f_2\colon Y_2\to X_2 $ such that $ K_{Y_2} + \Delta_2 \sim_\R f_2^*(K_{X_2} + B_2) $. In particular, the pair $ (Y_2,\Delta_2) $ is a dlt blowup of $ (X_2,B_2) $. By continuing this process analogously, we obtain the required diagram.
\end{proof}

The analogue of Lemma \ref{lem:lifting} in the context of g-pairs is the following:

\begin{lem}\label{lem:lifting_g}
	Assume the existence of minimal models for smooth varieties of dimension at most $ n-1 $.
	
	Let $ (X_1,B_1+M_1) $ be a quasi-projective NQC log canonical g-pair  of dimension $ n $ which is projective over a quasi-projective variety $Z$. Consider a sequence of small ample quasi-flips over $Z$:
	\begin{center}
		\begin{tikzcd}[column sep = 0.8em, row sep = 1.75em]
			(X_1,B_1+M_1) \arrow[dr, "\theta_1" swap] \arrow[rr, dashed, "\pi_1"] && (X_2,B_2+M_2) \arrow[dl, "\theta_1^+"] \arrow[dr, "\theta_2" swap] \arrow[rr, dashed, "\pi_2"] && (X_3,B_3+M_3) \arrow[dl, "\theta_2^+"] \arrow[rr, dashed, "\pi_3"] && \dots \\
			& Z_1 && Z_2
		\end{tikzcd}
	\end{center}
	Then there exists a diagram
	\begin{center}
		\begin{tikzcd}[column sep = 0.8em, row sep = large]
			(Y_1,\Delta_1+N_1) \arrow[d, "f_1" swap] \arrow[rr, dashed, "\rho_1"] && (Y_2,\Delta_2+N_2) \arrow[d, "f_2" swap] \arrow[rr, dashed, "\rho_2"] && (Y_3,\Delta_3+N_3) \arrow[d, "f_3" swap] \arrow[rr, dashed, "\rho_3"] && \dots 
			\\ 
			(X_1,B_1+M_1) \arrow[dr, "\theta_1" swap] \arrow[rr, dashed, "\pi_1"] && (X_2,B_2+M_2) \arrow[dl, "\theta_1^+"] \arrow[dr, "\theta_2" swap] \arrow[rr, dashed, "\pi_2"] && (X_3,B_3+M_3) \arrow[dl, "\theta_2^+"] \arrow[rr, dashed, "\pi_3"] && \dots \\
			& Z_1 && Z_2
		\end{tikzcd}
	\end{center}
	where, for each $i \geq 1$, the map $ \rho_i \colon Y_i \dashrightarrow Y_{i+1} $ is a $(K_{Y_i}+\Delta_i+N_i)$-MMP over $ Z_i $ and the map $f_i $ is a dlt blowup of the g-pair $ (X_i,B_i+M_i) $. 
	
	In particular, the sequence on top of the above diagram is an MMP for an NQC $\Q$-factorial dlt g-pair $ (Y_1,\Delta_1+N_1) $.
\end{lem}

\begin{proof}
	Let $ f_1 \colon (Y_1,\Delta_1+N_1) \to  (X_1,B_1+M_1) $ be a dlt blowup of the g-pair $ (X_1,B_1+M_1) $. By the definition of a small ample quasi-flip and by Lemma \ref{lem:discrep}, the g-pair $ (X_2,B_2+M_2) $ is a minimal model of $ (X_1,B_1+M_1) $ over $ Z_1 $, hence the g-pair $ (X_1,B_1+M_1) $ has a minimal model in the sense of Birkar-Shokurov over $ Z_1 $ by Remark \ref{rem:models}. Hence, $ (X_1,B_1+M_1) $ has an NQC weak Zariski decomposition over $ Z_1 $ by \cite[Proposition 5.1]{HL22}, and it follows by \cite[Remark 2.11]{LT22} that $ (Y_1,\Delta_1+N_1) $ has an NQC weak Zariski decomposition over $ Z_1 $. Therefore, by \cite[Theorem 4.4(ii)]{LT22} there exists a $ (K_{Y_1} + \Delta_1+N_1) $-MMP with scaling of an ample divisor over $Z_1$ which terminates with a minimal model $ (Y_2,\Delta_2+N_2) $ of $ (Y_1,\Delta_1+N_1) $ over $Z_1$. Since $(X_2,B_2+M_2)$ is a log canonical model of $(Y_1,\Delta_1+N_1)$ over $ Z_1 $, by Lemma \ref{lem:maptoCM} there exists a morphism $f_2\colon Y_2\to X_2 $ such that $ K_{Y_2} + \Delta_2 + N_2 \sim_\R f_2^*(K_{X_2} + B_2 + M_2) $. In particular, the g-pair $ (Y_2,\Delta_2 + M_2) $ is a dlt blowup of the g-pair $ (X_2,B_2+M_2) $. By continuing this process analogously, we obtain the required diagram.
\end{proof}

\section{Special termination for pairs}
\label{sect:SpecialTerm}

\begin{lem}\label{lem:technical}
Let $f\colon Y\to X$ be a projective birational morphism between normal varieties. Assume that we have a diagram
	\begin{center} 
		\begin{tikzcd}[column sep = 1em, row sep = normal]
			Y \arrow[d, "f" swap] \arrow[rr, dashed, "\mu"] && W \arrow[d] \\
			X \arrow[rr, "\theta"] && Z
		\end{tikzcd}
	\end{center}
where $ \theta $ is birational and $\mu$ is an isomorphism in codimension one. Let $D_X$ be an $\R$-Cartier $\R$-divisor on $X$, set $D_Y:=f^*D_X$ and $D_W:=\mu_*D_Y$, and assume that $D_W$ is $\R$-Cartier. Let $V_X\subseteq X$ and $V_Y\subseteq Y$ be closed subsets such that $f(V_Y)=V_X$. Assume that:
\begin{enumerate}
\item[(i)] the map $\mu$ is $D_Y$-nonpositive,
\item[(ii)] $D_W$ is semiample over $Z$, 
\item[(iii)] $V_Y$ is contained in the locus in $Y$ where the map $\mu$ is an isomorphism,
\item[(iv)] the set $\Exc(\theta)$ is covered by curves $\gamma$ which are contracted by $ \theta $ and such that $D_X\cdot\gamma<0$.
\end{enumerate}
Then $\Exc(\theta)\cap V_X=\emptyset$.
\end{lem}

\begin{proof}
	Arguing by contradiction, assume that there exists $ x \in \Exc(\theta) \cap V_X $ and set $ F := f^{-1}(x) $. We first claim that
	\begin{equation}\label{eq:6}
		F \subseteq \sB(D_Y/Z). 
	\end{equation}
	To this end, by (iv) we may find a curve $ \gamma \subseteq \Exc(\theta) $ passing through $ x $, contracted by 
	$ \theta $ and such that $ D_X \cdot \gamma < 0 $. But then for each $H \in | D_X /Z |_\R$ we have $ H \cdot \gamma < 0 $, and thus $ x \in \gamma \subseteq \Supp H $. This implies that $ x \in \sB(D_X/Z) $, and by Lemma \ref{lem:baselocus} we infer that 
	\[ F \subseteq f^{-1} \big( \sB(D_X/Z) \big) = \sB(D_Y/Z) , \]
	as desired.
		
	Now, since $ F \cap V_Y \neq \emptyset $, from \eqref{eq:6} we obtain
	\begin{equation}\label{eq:5}
		V_Y \cap \sB(D_Y/Z) \neq \emptyset .	
	\end{equation}	 
	Define $V_W:=\mu(V_Y)$ and note that $ V_W $ is well-defined by (iii). We claim that 
	$$ V_W \cap \sB(D_W/Z)\neq\emptyset ,$$ 
	which would then contradict (ii) and finish the proof. 
	
	To this end, by (iii) there exists a resolution of indeterminacies $ (p,q) \colon T \to Y\times W $ of the map $ \mu $ such that $p$ and $q$ are isomorphisms over some neighbourhoods of $V_Y$ and $V_W$, respectively. 
	\begin{center}
		\begin{tikzcd}
			& T \arrow[dl, "p" swap] \arrow[dr, "q"] \\
			Y \arrow[rr, dashed, "\mu"] && W
		\end{tikzcd}
	\end{center}
	Then by (i) there exists an effective $q$-exceptional $\R$-divisor $E_T$ on $T$ such that 
	\[ p^* D_Y \sim_\R q^* D_W + E_T . \]
	Fix $ G_W \in |D_W /Z |_\R $ and set
	$$G_Y:=p_*(q^* G_W + E_T)\in |D_Y/Z|_\R . $$ 
	We then have 
	$$ p^* G_Y = q^* G_W + E_T $$
	by the Negativity Lemma \cite[Lemma 3.39]{KM98}. Since $ V_Y \cap \Supp G_Y \neq\emptyset $ by \eqref{eq:5}, we obtain 
	$$ \emptyset\neq p^{-1}(V_Y) \cap \Supp (p^* G_Y)=p^{-1}(V_Y)\cap \Supp(q^* G_W + E_T) ,$$ 
	hence $ p^{-1}(V_Y) \cap q^{-1}(\Supp G_W)\neq\emptyset $, as $ p^{-1}(V_Y) $ does not intersect $ \Supp E_T $ by construction. Thus, as $V_W=q\big(p^{-1}(V_Y)\big)$, we have
	$$ V_W\cap \Supp G_W\neq\emptyset ,$$
	and the claim follows.
\end{proof}

\begin{proof}[Proof of Theorem \ref{thm:main}]
	By Lemma \ref{lem:klt_lc} we may assume the termination of flips for $ \Q $-factorial dlt pairs of dimension at most $ n-1 $. By Lemma \ref{lem:lifting} there exists a diagram
	\begin{center}
		\begin{tikzcd}[column sep = 2em, row sep = large]
			(Y_1,\Delta_1) \arrow[d, "f_1" swap] \arrow[rr, dashed, "\rho_1"] && (Y_2,\Delta_2) \arrow[d, "f_2" swap] \arrow[rr, dashed, "\rho_2"] && (Y_3,\Delta_3) \arrow[d, "f_3" swap] \arrow[rr, dashed, "\rho_3"] && \dots 
			\\ 
			(X_1,B_1) \arrow[dr, "\theta_1" swap] \arrow[rr, dashed, "\pi_1"] && (X_2,B_2) \arrow[dl, "\theta_1^+"] \arrow[dr, "\theta_2" swap] \arrow[rr, dashed, "\pi_2"] && (X_3,B_3) \arrow[dl, "\theta_2^+"] \arrow[rr, dashed, "\pi_3"] && \dots 
			\\
			& Z_1 && Z_2
		\end{tikzcd}
	\end{center}
	where the sequence of rational maps $\rho_i$ is a composition of steps of a 
	$ (K_{Y_1} + \Delta_1) $-MMP. By relabelling, we may assume that this MMP is a sequence of flips, and by \cite[Theorem 4.2.1]{Fuj07a} we may assume that the flipping locus avoids the non-klt locus at each step in this MMP. We conclude by applying Lemma \ref{lem:technical} for $X=X_1$, $Y=Y_1$, $D_X=K_{X_1}+B_1$, $D_Y=K_{Y_1}+\Delta_1$, $V_X=\nklt(X_1,B_1)$ and $V_Y=\nklt(Y_1,\Delta_1)$, taking Remark \ref{rem:nklt_dlt_blow-up} into account.
\end{proof}

\section{Special termination for g-pairs}

\begin{thm}\label{thm:specterm_g-pairs}
	Assume the termination of flips for NQC $ \Q $-factorial dlt g-pairs of dimension at most $ n-1 $. 
	
	Let $ (X_1,B_1+M_1) $ be a quasi-projective NQC $ \Q $-factorial dlt g-pair of dimension $ n $, which is projective over a normal quasi-projective variety $Z$. Consider a sequence of flips over $Z$:
	\begin{center}
		\begin{tikzcd}[column sep = 0.8em, row sep = 1.75em]
			(X_1,B_1+M_1) \arrow[dr, "\theta_1" swap] \arrow[rr, dashed, "\pi_1"] && (X_2,B_2+M_2) \arrow[dl, "\theta_1^+"] \arrow[dr, "\theta_2" swap] \arrow[rr, dashed, "\pi_2"] && (X_3,B_3+M_3) \arrow[dl, "\theta_2^+"] \arrow[rr, dashed, "\pi_3"] && \dots \\
			& Z_1 && Z_2
		\end{tikzcd}
	\end{center}
	Then there exists a positive integer $N$ such that 
	\[ \Exc(\theta_i)\cap\nklt(X_i,B_i+M_i)=\emptyset \text{ for all } i\geq N . \]
\end{thm}

\begin{proof}
	We follow closely the proofs of \cite[Theorem 4.5]{HL22} and \cite[Theorem 4.2.1]{Fuj07a}. We prove by induction on $d$ the following claim, and at the end of the proof we show how it implies the theorem.
	
	\medskip 
	
	\emph{Claim.} For each nonnegative integer $d$ there exists a positive integer $N_d$ such that the restriction of $\theta_i$ to each log canonical centre of dimension at most $d$ is an isomorphism for all $i\geq N_d$.
	
	\medskip
	
	To prove the Claim, recall first that the number of log canonical centres of any log canonical g-pair is finite. At step $i$ of the MMP as above, if a log canonical centre of $(X_i,B_i+M_i)$ belongs to $\Exc(\theta_i)$, then the number of log canonical centres of $(X_{i+1},B_{i+1}+M_{i+1})$ is smaller than the number of log canonical centres of $(X_i,B_i+M_i)$ by Lemma \ref{lem:discrep}. 
	
	Thus, there exists a positive integer $N_0$ such that the set $\Exc(\theta_i)$ does not contain any log canonical centre of $(X_i,B_i+M_i)$ for $i\geq N_0$. By relabelling, we may assume that $N_0=1$. In particular, this proves the Claim for $d=0$.
		
	Therefore, we may assume that for each $i\geq1$, the map $\pi_i$ is an isomorphism at the generic point of each log canonical centre of $(X_i,B_i+M_i)$.
	
	Let $d$ be a positive integer. By induction and by relabelling, we may assume that each map $ \pi_i $ is an isomorphism along every log canonical centre of dimension at most $ d -1 $. 
	
	Now, we consider a log canonical centre $ S_1 $ of $ (X_1,B_1+M_1) $ of dimension $ d $. We obtain a sequence of birational maps $\pi_i|_{S_i}\colon S_i\dashrightarrow S_{i+1}$, where $ S_i $ is the strict transform of $ S_1 $ on $ X_i $. Every log canonical centre of the NQC dlt g-pair $ (S_i,B_{S_i} + M_{S_i}) $ is a log canonical centre of $(X_i,B_i + M_i)$, and hence by induction, each map $\pi_i$ is an isomorphism along $\Supp \lfloor B_{S_i}\rfloor$. Then by Proposition \ref{pro:monoton} and since the difficulty takes values in $\N$, after relabelling the indices we may assume that $ S_i $ and $ S_{i+1} $ are isomorphic in codimension $1$ for every $ i $. 
	
	Moreover, by relabelling the indices we may assume that $(\pi_i|_{S_i})_*B_{S_i}=B_{S_{i+1}}$: indeed, this is equivalent to saying that we have 
	\begin{equation}\label{eq:79}
	a(E,S_i,B_{S_i} +M_{S_i}) = a(E,S_{i+1},B_{S_{i+1}}+M_{S_{i+1}})	
	\end{equation}
	for each component of $B_{S_i}$ and $B_{S_{i+1}}$. Since $\pi_i$ is an isomorphism along $\Supp \lfloor B_{S_i}\rfloor$, the equation \eqref{eq:79} is clear if $E\subseteq \Supp \lfloor B_{S_i}\rfloor$. Note that in general we have $a(E,S_i,B_{S_i} +M_{S_i}) \leq a(E,S_{i+1},B_{S_{i+1}}+M_{S_{i+1}})$ by Lemma \ref{lem:discrep}. If $E\nsubseteq \Supp \lfloor B_{S_i}\rfloor$, then by Lemma \ref{lem:finite} there exists a finite subset $\Gamma\subseteq\mathcal S(B,\mu)$, which is independent of the index $i$, such that $a(E,S_i,B_{S_i} +M_{S_i})\in\Gamma$. Therefore, after relabelling the indices we may assume that \eqref{eq:79} holds.
		
	For every $ i \geq 1 $ we denote by $ T_i $ the normalisation of $ \theta_i(S_i) $. By Lemma \ref{lem:lifting_g} there exists a diagram
	\begin{center}
		\begin{tikzcd}[column sep = 0.65em, row sep = 2.5em]
			(W_1,\Delta_1+N_1) \arrow[d, "f_1" swap] \arrow[rr, dashed, "\rho_1"] && (W_2,\Delta_2+N_2) \arrow[d, "f_2" swap] \arrow[rr, dashed, "\rho_2"] && (W_3,\Delta_3+N_3) \arrow[d, "f_3" swap] \arrow[rr, dashed, "\rho_3"] && \dots 
			\\ 
			(S_1,B_{S_1}+M_{S_1}) \arrow[dr, "\theta_1|_{S_1}" swap] \arrow[rr, dashed, "\pi_1|_{S_1}"] && (S_2,B_{S_2}+M_{S_2}) \arrow[dl, "\theta_1^+|_{S_2}"] \arrow[dr, "\theta_2|_{S_2}" swap] \arrow[rr, dashed, "\pi_2|_{S_2}"] && (S_3,B_{S_3}+M_{S_3}) \arrow[dl, "\theta_2^+|_{S_3}"] \arrow[rr, dashed] && \dots 
			\\
			& T_1 && T_2
		\end{tikzcd}
	\end{center}
	where the sequence of rational maps $\rho_i$ yields an MMP for the NQC $ \Q $-factorial dlt g-pair $ (W_1, \Delta_1+N_1) $. By the assumptions of the theorem, this MMP terminates, so by relabelling, we may assume that 
	$$ (W_i,\Delta_i+N_i) = (W_{i+1}, \Delta_{i+1} + N_{i+1}) \quad \text{for all } i\geq1 .$$
	Since $-(K_{W_i}+\Delta_i+N_i)$ is nef over $T_i$ and $K_{W_{i+1}}+\Delta_{i+1}+N_{i+1}$ is nef over $T_i$ by construction, we obtain that $K_{W_i}+\Delta_i+N_i$ is numerically trivial over $T_i$ for each $i$. In particular, $K_{S_i}+B_{S_i}+M_{S_i}$ and $K_{S_{i+1}}+B_{S_{i+1}}+M_{S_{i+1}}$ are numerically trivial over $T_i$ for each $i$, and thus $\theta_i |_{S_i}$ and $\theta_i^+ |_{S_{i+1}}$ contract no curves. Therefore, $\theta_i |_{S_i}$ and $\theta_i^+ |_{S_{i+1}}$ are isomorphisms, and consequently all maps $\pi_i|_{S_i}$ are isomorphisms. This finishes the proof of the Claim.
	
	Finally, we show that the Claim implies the Theorem: indeed, the Claim shows that $\lfloor B_i\rfloor$ does not contain any flipping or flipped curves for all $i\geq N_{n-1}$. Thus, if $\Exc(\theta_i)\cap\lfloor B_i\rfloor\neq\emptyset$ for some $i\geq N_{n-1}$, then there is a curve $C$ contracted by $\theta_i$ with $C\cdot\lfloor B_i\rfloor > 0$. But then $C^+\cdot\lfloor B_{i+1}\rfloor < 0$ for every curve $C^+$ contracted by $\theta_i^+$, hence $C^+\subseteq\lfloor B_{i+1}\rfloor$, a contradiction.
\end{proof}

The analogue of Lemma \ref{lem:klt_lc} in the context of g-pairs is the following: 

\begin{lem} \label{lem:reduction_g-term}
	The termination of flips for quasi-projective NQC $\Q$-factorial klt g-pairs of dimension at most $ d $ which are projective over a normal quasi-projective variety $Z$ implies the termination of flips for quasi-projective NQC log canonical g-pairs of dimension $ d $ over $Z$.
\end{lem}

\begin{proof}
	By induction, we may assume the termination of flips for NQC $ \Q $-factorial dlt g-pairs of dimension at most $ d-1 $. 

	Consider a sequence of flips starting from an NQC log canonical g-pair 
	$ (X_1,B_1+M_1) $ of dimension $ d $:
	\begin{center}
		\begin{tikzcd}[column sep = 0.8em, row sep = 1.75em]
			(X_1,B_1+M_1) \arrow[dr, "\theta_1" swap] \arrow[rr, dashed, "\pi_1"] && (X_2,B_2+M_2) \arrow[dl, "\theta_1^+"] \arrow[dr, "\theta_2" swap] \arrow[rr, dashed, "\pi_2"] && (X_3,B_3+M_3) \arrow[dl, "\theta_2^+"] \arrow[rr, dashed, "\pi_3"] && \dots \\
			& Z_1 && Z_2
		\end{tikzcd}
	\end{center}
	By Lemma \ref{lem:lifting_g} there exists a diagram
	\begin{center}
		\begin{tikzcd}[column sep = 0.8em, row sep = large]
			(Y_1,\Delta_1+N_1) \arrow[d, "f_1" swap] \arrow[rr, dashed, "\rho_1"] && (Y_2,\Delta_2+N_2) \arrow[d, "f_2" swap] \arrow[rr, dashed, "\rho_2"] && (Y_3,\Delta_3+N_3) \arrow[d, "f_3" swap] \arrow[rr, dashed, "\rho_3"] && \dots 
			\\ 
			(X_1,B_1+M_1) \arrow[dr, "\theta_1" swap] \arrow[rr, dashed, "\pi_1"] && (X_2,B_2+M_2) \arrow[dl, "\theta_1^+"] \arrow[dr, "\theta_2" swap] \arrow[rr, dashed, "\pi_2"] && (X_3,B_3+M_3) \arrow[dl, "\theta_2^+"] \arrow[rr, dashed, "\pi_3"] && \dots \\
			& Z_1 && Z_2
		\end{tikzcd}
	\end{center}
	where the sequence of rational maps $\rho_i$ is a composition of steps in an MMP for an NQC $\Q$-factorial dlt g-pair $ (Y_1,\Delta_1 + N_1) $. It suffices to show that this MMP terminates; we may assume that this sequence consists only of flips. By Theorem \ref{thm:specterm_g-pairs} and by relabelling, we may also assume that in this sequence the flipping locus at each step avoids the non-klt locus. Consequently, this sequence of flips is also a sequence of flips for the NQC $\Q$-factorial klt g-pair $ \big(Y_1,(\Delta_1 - \lfloor \Delta_1 \rfloor) +N_1 \big) $, which terminates by assumption.
\end{proof}

Finally, we obtain the analogue of Theorem \ref{thm:main} in the context of g-pairs.

\begin{proof}[Proof of Theorem \ref{thm:main_g}]
The proof is analogous to that of Theorem \ref{thm:main}, by replacing Lemma \ref{lem:klt_lc} by Lemma \ref{lem:reduction_g-term}, Lemma \ref{lem:lifting} by Lemma \ref{lem:lifting_g}, and \cite[Theorem 4.2.1]{Fuj07a} by Theorem \ref{thm:specterm_g-pairs}.
\end{proof}
 
	\bibliographystyle{amsalpha}
	\bibliography{biblio}

\newcommand{\etalchar}[1]{$^{#1}$}
\providecommand{\bysame}{\leavevmode\hbox to3em{\hrulefill}\thinspace}
\providecommand{\MR}{\relax\ifhmode\unskip\space\fi MR }
\providecommand{\MRhref}[2]{%
  \href{http://www.ams.org/mathscinet-getitem?mr=#1}{#2}
}
\providecommand{\href}[2]{#2}
\begin{thebibliography}{Mor18}

\bibitem[Bir07]{Bir07}
C.~Birkar, \emph{Ascending chain condition for log canonical thresholds and
  termination of log flips}, Duke Math. J. \textbf{136} (2007), no.~1,
  173--180.

\bibitem[Bir12]{Bir12a}
\bysame, \emph{Existence of log canonical flips and a special {LMMP}}, Publ.
  Math. Inst. Hautes \'Etudes Sci. \textbf{115} (2012), no.~1, 325--368.

\bibitem[BZ16]{BZ16}
C.~Birkar and D.-Q. Zhang, \emph{Effectivity of {I}itaka fibrations and
  pluricanonical systems of polarized pairs}, Publ. Math. Inst. Hautes \'Etudes
  Sci. \textbf{123} (2016), 283--331.

\bibitem[Deb01]{Deb01}
O.~Debarre, \emph{Higher-dimensional algebraic geometry}, Universitext,
  Springer-Verlag, New York, 2001.

\bibitem[Fuj07]{Fuj07a}
O.~Fujino, \emph{Special termination and reduction to pl flips}, Flips for
  3-folds and 4-folds (Alessio Corti, ed.), Oxford Lecture Ser. Math. Appl.,
  vol.~35, Oxford Univ. Press, Oxford, 2007, pp.~63--75.

\bibitem[Fuj11]{Fuj11}
\bysame, \emph{Recent developments in minimal model theory}, Sugaku Expositions
  \textbf{24} (2011), no.~2, 205--237.

\bibitem[Fuj17]{Fuj17}
\bysame, \emph{Foundations of the minimal model program}, MSJ Memoirs, vol.~35,
  Mathematical Society of Japan, Tokyo, 2017.

\bibitem[Has18]{Hash18a}
K.~Hashizume, \emph{Minimal model theory for relatively trivial log canonical
  pairs}, Ann. Inst. Fourier \textbf{68} (2018), no.~5, 2069--2107.

\bibitem[Has19]{Hash19b}
\bysame, \emph{Remarks on special kinds of the relative log minimal model
  program}, Manuscripta Math. \textbf{160} (2019), no.~3-4, 285--314.

\bibitem[HL22]{HL22}
J.~Han and Z.~Li, \emph{Weak {Z}ariski decompositions and log terminal models
  for generalized pairs}, Math. Z. \textbf{302} (2022), no.~2, 707--741.

\bibitem[HM20]{HM20}
C.~D. Hacon and J.~Moraga, \emph{On weak {Z}ariski decompositions and
  termination of flips}, Math. Res. Lett. \textbf{27} (2020), no.~5,
  1393--1422.

\bibitem[HX13]{HX13}
C.~D. Hacon and C.~Xu, \emph{Existence of log canonical closures}, Invent.
  Math. \textbf{192} (2013), no.~1, 161--195.

\bibitem[K{\etalchar{+}}92]{Kol92}
J.~Koll{\'a}r et~al., \emph{Flips and abundance for algebraic threefolds},
  Ast\'erisque 211, Soc.\ Math.\ France, Paris, 1992.

\bibitem[KM98]{KM98}
J.~Koll{\'a}r and S.~Mori, \emph{Birational geometry of algebraic varieties},
  Cambridge Tracts in Mathematics, vol. 134, Cambridge University Press,
  Cambridge, 1998.

\bibitem[LT22]{LT22}
V.~Lazi\'c and N.~Tsakanikas, \emph{On the existence of minimal models for log
  canonical pairs}, Publ. Res. Inst. Math. Sci. \textbf{58} (2022), no.~2,
  311--339.

\bibitem[Mor18]{Mor18}
J.~Moraga, \emph{{Termination of pseudo-effective 4-fold flips}},
  arXiv:1802.10202\setbox0=\hbox{2018}.

\bibitem[Sho85]{Sho85}
V.~V. Shokurov, \emph{A nonvanishing theorem}, Izv. Akad. Nauk SSSR Ser. Mat.
  \textbf{49} (1985), no.~3, 635--651.

\bibitem[Sho92]{Sho92}
\bysame, \emph{Three-dimensional log perestroikas}, Izv. Ross. Akad. Nauk Ser.
  Mat. \textbf{56} (1992), no.~1, 105--203.

\bibitem[Sho03]{Sho03}
\bysame, \emph{Prelimiting flips}, Tr. Mat. Inst. Steklova \textbf{240} (2003),
  82--219.

\bibitem[Sho04]{Sho04}
\bysame, \emph{Letters of a bi-rationalist. {V}. {M}inimal log discrepancies
  and termination of log flips}, Tr. Mat. Inst. Steklova \textbf{246} (2004),
  no.~Algebr. Geom. Metody, Svyazi i Prilozh., 328--351.

\end{thebibliography}
	
\end{document}